\documentclass[12pt]{amsart}

\usepackage{amssymb}
\usepackage{graphicx}
\usepackage{epstopdf}
\usepackage{caption}
\usepackage{calc}
\usepackage{enumitem}
\usepackage{url}
\usepackage{mathrsfs}

\usepackage{soul}
\usepackage{color}
\definecolor{cyan(process)}{rgb}{0.0, 0.72, 0.92}

\theoremstyle{plain}
\newtheorem{theorem}{Theorem}[section]
\newtheorem{prop}[theorem]{Proposition}
\newtheorem*{theorem*}{Theorem}
\newtheorem*{prop*}{Proposition}

\theoremstyle{definition}
\newtheorem{definition}[theorem]{Definition}
\newtheorem{remark}[theorem]{Remark}

\newtheorem{lemma}[theorem]{Lemma}
\newtheorem*{ex}{Example}

\newcommand{\Z}{\mathbb{Z}}

\title{Palindromic sequences of the Markov spectrum}

\author{Matty van Son}
\date{01 September 2018}

\address{Matty van-Son\\
University of Liverpool\\
Mathematical Sciences Building\\
Liverpool L69 7ZL, United Kingdom
} \email{sgmvanso@liverpool.ac.uk}


\begin{document}

	\begin{abstract}
		We study the periods of Markov sequences, which are derived from the continued fraction expression of elements in the Markov spectrum. This spectrum is the set of minimal values of indefinite binary quadratic forms that are specially normalised. We show that the periods of these sequences are palindromic after a number of circular shifts, the number of shifts being given by Stern's diatomic sequence.
	\end{abstract}
	
	\maketitle
	
	\section*{Introduction}
	In this paper we state and prove a general result on the construction of palindromic sequences. These include sequences relating to the Markov spectrum. The \textit{Markov spectrum} is the set of numbers
	\begin{equation} \label{eq inf}
	\inf\limits_{\Z^2\setminus \{(0,0)\}}\bigg|\frac{\sqrt{\Delta}}{f}\bigg|
	\end{equation}
	for all binary quadratic forms $f$ with positive discriminant $\Delta(f)$.
	
	A. Markov showed in his papers \cite{mar1,mar2} that for any element of the Markov spectrum less than $3$ there exists a sequence of positive integers $(a_1,\ldots,a_{2n})$ such that 
	\begin{equation} \label{eq PI}
	\inf\limits_{\Z^2\setminus \{(0,0)\}}\bigg|\frac{\sqrt{\Delta}}{f}\bigg|=[(a_1,a_2,\ldots,a_{2n})]+[0;\overline{(a_1,\ldots,a_{2n})}],
	\end{equation}
	where $[(a_1,a_2,\ldots,a_{2n})]$ is the infinite continued fraction with~period $(a_1,a_2,\ldots,a_{2n})$. In this paper we denote the reverse of the sequence $(a_1,\ldots,a_{2n})$ by $\overline{(a_1,\ldots,a_{2n})}$. Equation (\ref{eq PI}) is known as the \textit{Perron identity}, going back to \cite{per1}. 
	
	It is known, see for example the books by T. Cusick and M. Flahive \cite{cus1}, and M. Aigner \cite{aigner}, that the numbers $a_i$ satisfy the following
	\begin{itemize}
	\item $a_i\in\{1,2\}$
	\item $a_1=a_2=2$, $a_{2n}=a_{2n-1}=1$
	\item The subsequence $w{=}(a_3,\ldots,a_{2n-2})$ is palindromic, i.e. $w=\overline{w}$.
	\end{itemize}
	One can find work on Markov numbers in relation to other branches of mathematics in the following articles: \cite{oleg1}, \cite{series}, \cite{reutenauer}, \cite{oleg&matty}, \cite{veselov&spalding}.
	
	The sequences for which expression (\ref{eq inf}) is less than $3$, henceforth called \textit{Markov sequences}, may be constructed by concatenation of the sequences $(2,2)$, and $(1,1)$ (see Definition \ref{def ms} below). This follows as a corollary of the work of H.~Cohn~\cite{cohn}, specifically found in \cite[Theorem~4.7]{aigner}. 
	
	We show in Theorem \ref{thm main} that any sequence constructed in the same way as Markov sequences  are \textit{evenly palindromic}, that is, after some number of circular shifts the sequence is palindromic. The number of circular shifts is given by Stern's diatomic sequence, an exposition of which can be found in the paper of I. Urbiha, \cite{urbiha}.
	
	In the forthcoming paper \cite{oleg&matty2}, we use Theorem \ref{thm main} to show that there is a generalisation of Markov numbers coming from the graph structure in Definition \ref{def ms}.

	\subsection*{Organisation of the Paper}
	In Section \ref{sec1} we give a background for Markov sequences, and give the necessary definitions for our main result, Theorem \ref{thm main}.
	
	In Section \ref{s lem and proof} we prove Theorem \ref{thm main}.
	
	\noindent \textbf{Acknowledgements.} The author is grateful to O. Karpenkov for his constant attention to this work.

	\section{Some history and background} \label{sec1}
	In this section we give the necessary definitions and background for the main result of the paper, Theorem \ref{thm main}. 
	
	\subsection{The Markov spectrum}
	In this subsection we define the Markov spectrum in terms of binary quadratic forms and sequences of positive integers.
	We start with the following definition.
	\begin{definition} \label{def ms1}
		Let $f$ be a binary quadratic form with positive discriminant $\Delta$. The \textit{Markov element of $f$} is defined to be
		\[
		M(f)=\inf\limits_{\Z^2\setminus \{(0,0)\}}\left|\frac{f}{\sqrt{\Delta}} \right|.
		\]
		The \textit{Markov spectrum} is the set of values $1/M(f)$ for all such forms $f$.
	\end{definition}
	For a sequence of positive integers $a_1,\ldots,a_n$, let 
	\[
	[a_1;a_2:\ldots:a_n]
	\]
	denote the continued fraction of $a_1,\ldots,a_n$.
	We give an alternative definition of the Markov spectrum.
	\begin{definition} \label{def ms2}
		Let $A$ be a doubly infinite sequence of positive integers
		\[
		A=\ldots,a_{-2},a_{-1},a_0,a_1,a_2,\ldots.
		\]
		Define $M(A)$, the \textit{Markov element of $A$}, by
		\begin{equation} \label{pi}
		\frac{1}{M(A)}=\inf\limits_{i\in\Z}\left(a_i+[0;a_{i+1}:a_{i+2}:\ldots]+[0;a_{i-1}:a_{i-2}:\ldots]\right).
		\end{equation}
		The right hand side of Equation (\ref{pi}) is known as the \textit{Perron identity}. The set of values $1/M(A)$ for all such sequences $A$ is called the \textit{Markov spectrum}.
	\end{definition}
	
	\begin{remark}
		Definitions \ref{def ms1} and \ref{def ms1} are equivalent, see \cite{per1}. The sequences $A$ for which $M(A)>1/3$ are purely periodic and consist solely of the integers $1$, and $2$. We refer to these sequences as \textit{Markov sequences}.
	\end{remark}
	
	\subsection{Graph structure of Markov sequences}
	We give an alternative definition of Markov sequences.
	\begin{definition}
		Let $\Z^{\infty}$ be the set of finite sequences of integer elements. Consider the binary operation $\oplus$ on $\Z^{\infty}$ defined as
		\[
		(a_1,\ldots,a_n)\oplus (b_1,\ldots,b_m)=(a_1,\ldots,a_n,b_1,\ldots,b_m).
		\]
		We call this the \textit{concatenation} of sequences $(a_1,\ldots,a_n)$ and $(b_1,\ldots,b_m)$.
		Let also for $A$, $B$, $C\in\Z^{\infty}$
		\[
		\mathcal{L}_{\oplus}(A,B,C)=(A,A\oplus B,B),\ \ 	\mathcal{R}_{\oplus}(A,B,C)=(B,B\oplus C,C).
		\]
	\end{definition}

	\begin{definition} \label{def ms}
		We define $\mathcal{G}(\Z^{\infty},\oplus,x)$ to be the directed graph whose vertices are elements in $({\Z^{\infty}})^3$, and containing the vertex $x$. The vertices $v$, $w\in{\Z^{\infty}}^3$ are connected by an edge $(v,w)$ if either
		\[
		w=\mathcal{L}_{\oplus}(v),\ \ \ \mbox{or}\ \ \ w=\mathcal{R}_{\oplus}(v).
		\]
		
		For $A$, $B\in\Z^{\infty}$ we write
		\[
		\mathcal{G}(\Z^{\infty},\oplus,(A,A\oplus B,B)=\mathcal{G}_{A,B}.
		\]
	\end{definition}
	\begin{remark}
		The graph $\mathcal{G}_{(1,1),(2,2)}$ is called the \textit{graph of general Markov sequences} and contains all Markov sequences, see \cite{cus1,oleg&matty2}. 
	\end{remark}

	\begin{definition}
		Let $v$ be the vertex $(A,A\oplus B,B))\in\mathcal{G}_{A,B}$. Let $w$ be a vertex in $\mathcal{G}_{A,B}$. We say that $(\alpha_1,\ldots,\alpha_{2n})$ is a path from $v$ to $w$ if
		\[
		w=\mathcal{L}_{\oplus}^{\alpha_{2n}}\mathcal{R}_{\oplus}^{\alpha_{2n-1}}\ldots \mathcal{L}_{\oplus}^{\alpha_2}\mathcal{R}_{\oplus}^{\alpha_1}(v).
		\]
		
		We define the $N$-th level in $\mathcal{G}_{A,B}$ to be all vertices $w$ such that the path $(\alpha_1,\ldots,\alpha_{2n})$ from $v$ to $w$ satisfies
		$
		\sum_{i=1}^{2n}\alpha_i=N.
		$
	\end{definition}
	
	We give an ordering for the vertices in each level of $\mathcal{G}_{A,B}$.
	\begin{definition}
		For positive integers $n$, $m$, $\alpha_1,\ldots,\alpha_{2n}$, $\beta_1,\ldots,\beta_{2m}$ satisfying
		$\sum_{i=1}^{2n}\alpha_i=\sum_{i=1}^{2m}\beta_i$, let $w_1$, $w_2$ be two vertices in $\mathcal{G}_{A,B}$ with
		\[
		w_1=\mathcal{L}_{\oplus}^{\alpha_{2n}}\mathcal{R}_{\oplus}^{\alpha_{2n-1}}\ldots \mathcal{L}_{\oplus}^{\alpha_2}\mathcal{R}_{\oplus}^{\alpha_1}(v)\ \mbox{and}\ w_2=\mathcal{L}_{\oplus}^{\beta_{2m}}\mathcal{R}_{\oplus}^{\beta_{2m-1}}\ldots \mathcal{L}_{\oplus}^{\beta_2}\mathcal{R}_{\oplus}^{\beta_1}(v).
		\]
		Define an ordering of vertices by
		\[
		w_1\prec w_2 
		\]
		if either
		\[
		\begin{aligned}
		\alpha_i&=\beta_i, \mbox{ for}\ i=1,\ldots,2k{-}1,\ k<m,n\ \mbox{and } \alpha_{2k}>\beta_{2k},\ \mbox{or}\\
		\alpha_i&=\beta_i, \mbox{ for}\ i=1,\ldots,2k,\ k<m,n\ \mbox{and } \alpha_{2k+1}<\beta_{2k+1}.
		\end{aligned}	
		\]
	\end{definition}
	\begin{definition}
		Let the pair $(\mathcal{G}_{A,B},\prec )$ be the graph $\mathcal{G}_{A,B}$ where each level $n$ is ordered 
		\[
		w_1\prec w_2\prec \ldots\prec w_{2^n}.
		\]
	\end{definition}
	
	We define the sequence $(S(i))_{i=0}^{\infty}$.
	\begin{definition} \label{def first s}
		For two sequences $A$, and $B$ let  
		\[
		S(0)=A,\ \ S(1)=B\ \ S(2)=A\oplus B.
		\]
		For $n>1$, and $1\leq i\leq 2^{n-1}$ let $S(2^{n-1}{+}i)$ be the central element of the $i$-th vertex in the $n$-th level of the ordered graph $(\mathcal{G}_{A,B},\prec )$. We call $(S(i))$ the \textit{ordered Markov sequences for $A$ and $B$.}
		
		When we want to specify the sequences $A$, and $B$ we write
		\[
		S_{A,B}(n).
		\]
	\end{definition}
	
	\begin{ex}
		We have 
		\[
		S_{(a,a),(b,b)}(14)=(a,a,b,b,a,a,b,b,b,b,a,a,b,b,b,b).
		\]
		For $(a,b)=(1,2)$ we have for all $i\geq1$ that $M(S_{(a,a),(b,b)}(i))>1/3$.
	\end{ex}
	
	\begin{definition}
		Let $\Lambda=(\lambda_1,\ldots,\lambda_{2n})$, $\Gamma=(\gamma_1,\ldots,\gamma_{2n+1})$. We call $\Lambda$ and $\Gamma$ \textit{evenly palindromic} and \textit{oddly palindromic} respectively if there exists $k_1$, $k_2\in\Z$ such that for all $i\in\Z$ we have
		\[
		\begin{aligned}
		\lambda_{k_1+i\mod 2n}&=\lambda_{k_1-i-1\mod 2n},\\
		\gamma_{k_2+i+1\mod 2n+1}&=\gamma_{k_2-i-1\mod 2n+1}.
		\end{aligned}
		\]
	\end{definition}
	\begin{definition} \label{def sds}
		Let $d_0=0$ and $d_1=1$, and for all positive integers $n>1$ set
		\[
		d_{2n}=d_n,\ \ \ \		d_{2n{-}1}=d_n{+}d_{n{-}1}.
		\]
		The sequence $(d_n)_{n\geq0}$ is called \textit{Stern's diatomic sequence}.
	\end{definition}
	
	We give the main theorem of this paper.
	\begin{theorem} \label{thm main}
		Let $A$ and $B$ be two palindromic sequences of positive integers. Let $n$ be a positive integer, and let $N$ be the sum of powers of $A$'s and $B$'s in $S_{A,B}(n)$. Let $\Lambda_i\in\{A,B\}$ such that
		\[
		S_{A,B}(n)=\Lambda_1\ldots\Lambda_N.
		\]
		Then the following sequences are palindromic
		\[
		\begin{cases}
		\Lambda_{\lceil d_n/2\rceil}\Lambda_{\lceil d_n/2\rceil+1} \ldots\Lambda_N\Lambda_1\ldots \Lambda_{\lceil d_n/2\rceil-1}, & d_n \mbox{ even},\\
		\lceil\Lambda_{\lceil d_n/2\rceil}\rceil\Lambda_{\lceil d_n/2\rceil+1} \ldots\Lambda_N\Lambda_1\ldots \Lambda_{\lceil d_n/2\rceil-1}\lfloor\Lambda_{\lceil d_n/2\rceil}\rfloor, & d_n \mbox{ odd}.
		\end{cases}
		\]
	\end{theorem}

	\section{Proof of Theorem \ref{thm main}} \label{s lem and proof}
	In this section we prove Theorem \ref{thm main}. We start by stating Proposition \ref{lem main}, which deals with the majority of the proof. Subsections \ref{ssec lem main} through \ref{ssec proof} deal with proving this lemma, while the final proof of Theorem \ref{thm main} is in Subsection \ref{ssec final proof}.
	
	\subsection{Alternative definition for $S(n)$} \label{ssec lem main}
	We state Proposition \ref{lem main}, which is central to our proof of Theorem \ref{thm main}.	We then give an alternative definition for the sequences $(S(n))$ in Proposition \ref{seq def}, defining each $S(n)$ by concatenations of previous terms in $(S(n))$. We start by defining circular shifts of Markov sequences.
	
	\begin{definition}
		Let $A=(a_0,\ldots,a_n)$ be a sequence of positive integers, and for each $0\leq i<n$ define the operation 
		\[
		C_i(A)=(a_i,\ldots,a_n,a_1,\ldots,a_{i-1}).
		\]
		Then $C_i(A)$ is called the \textit{$i$-th circular shift of $A$}.
	\end{definition}
	
	\begin{prop} \label{lem main}
		Every sequence $S(i)$ is evenly palindromic. Moreover, we have that the sequence
		\[
		C_{d_i}(S(i))
		\]
		is palindromic, where $d_i$ is the $i$-th element in Stern's diatomic sequence.
	\end{prop}
	
	\begin{ex}
		$S_{(a,a),(b,b)}(3)=(a,a,a,a,b,b)$, and $d_3=2$. Then 
		\[
		C_2(S_{(a,a),(b,b)}(3))=(a,a,b,b,a,a).
		\]
	\end{ex}
	
	We define a sequence $(a(j))_{j\geq0}$ which simplifies notation of $(S(j))$.
	\begin{definition} \label{def a}
		Let $a(1){=}a(2){=}1$, and for all positive integers $j>1$ set
		\[
		a(2j)=a(j)\ \ \ \ 		a(2j{-}1)=j.
		\]
		The sequence $(a(j))_{j\geq0}$ is A003602 in \cite{seq}.
		Denote by $(a^*(j))$ the sequence defined 
		\[
		a^*(j)=
		\begin{cases}
		a(j), & \mbox{if } j>1,\\
		0, & \mbox{if } j=1.\\
		\end{cases}
		\]
		For the $a(j)$-th element in the sequence $(S(n))$, we write $S\circ a(j)$.
	\end{definition}

	\begin{ex}
		The first 10 elements of $(a(j))$ and $(d_j)$ are given
		\[
		\begin{aligned}
		(a(j))_1^{10}&=(1, 1, 2, 1, 3, 2, 4, 1, 5, 3),\\
		(d_j)_0^9&=(0, 1, 1, 2, 1, 3, 2, 3, 1, 4).
		\end{aligned}
		\]
		Figure \ref{fig sds} shows the symmetry in Stern's diatomic sequence.
	\end{ex}
	\begin{figure} 
		\includegraphics{sds-2.mps}
		\caption{The first entries in Stern's diatomic sequence.} \label{fig sds}
	\end{figure}
	
	\begin{prop} \label{seq def}
		Let $a>b$ be positive integers, and let
		\[
		S(0)=A,\ \ S(1)=B,\ \ S(2)=A\oplus B.
		\]
		The following definitions of the sequence $(S(j))_{j>2}$ are equivalent.
		\begin{itemize}
			\item[(i)] For $n>1$ and $1\leq i\leq 2^{n-1}$, define $S(2^{n-1}{+}i)$ to be the central element of the $i$-th vertex in the $n$-th level of the graph
			\[
			\mathcal{G}(\Z^{\infty},\oplus,(A,A\oplus B,B).
			\]
			\item[(ii)] Define
			\[
			\begin{aligned}
			S(2j)&=S(j)\oplus S\circ a(j),\ \mbox{and}\\ 
			S(2j{-}1)&=S\circ a^*(j{-}1)\oplus S(j).
			\end{aligned}
			\]
		\end{itemize}
	\end{prop}
	\begin{remark}
		Proposition \ref{seq def} gives us an alternative definition for the sequences $(S(n))$.
	\end{remark}
	For Proposition \ref{seq def} we first prove the following lemma.
	\begin{lemma} \label{lem must show}
		For $m=2k>2$ the following equations hold
		\[
		\begin{aligned}
		S(2^{n-2}{+}k)&=S( a^*(2^{n-1}{+}m{-}1)),\	S(2^{n-1}{+}2k)=S(2^{n-1}{+}m),\\
		S(2(2^{n-2}{+}k))&=S(2^{n-1}{+}m),\	S( a(2^{n-2}{+}k))=S( a(2^{n-1}{+}m)).
		\end{aligned}
		\]
		For $m=2k{-}1>2$ the following equations hold
		\[
		\begin{aligned}
		S( a^*(2^{n-2}{+}k{-}1))&=S( a^*(2^{n-1}{+}m{-}1)),\ 	S(2^{n-2}{+}k)=S( a(2^{n-1}{+}m)).\\
		S(2(2^{n-2}{+}k){-}1)&=S(2^{n-1}{+}m),\	\ S(2(2^{n-2}{+}k){-}1)=S(2^{n-1}{+}m),
		\end{aligned}
		\]
	\end{lemma}
	
	\begin{proof}
		For each case the equations follow from direct application of Definition \ref{def a}.
	\end{proof}
	
	\begin{proof}[Proof of Proposition \ref{seq def}]
		We prove this by induction on the levels of the graph of Markov sequences. The base of induction is given by
		\[
		\begin{aligned}
		S(3)&=S(0)\oplus S(2)=S\circ a^*(1)\oplus S(2),\ \mbox{and}\\
		S(4)&=S(2)\oplus S(1)=S(2)\oplus S\circ a(2).
		\end{aligned}
		\]
		
		Next we assume that the hypothesis is true for every $S(j)$ up the $n{-}1$-th level, that is to say, we have for $i=1,\ldots,2^{n-1}$ that  
		\begin{equation} \label{eq to start}
		S(2^{n-1}{+}i)=
		\begin{cases}
		S(2(2^{n-2}{+}k)) , & \mbox{if } i=2k,\\
		S(2(2^{n-2}{+}k){-}1), & \mbox{if } i=2k-1.
		\end{cases}	
		\end{equation}
		By definition of $(S(j))$ we have that
		\[
		\begin{aligned}
		S(2(2^{n-2}{+}k))=S(2^{n-2}{+}k)\oplus S\circ a(2^{n-2}{+}k),\ \ \mbox{and}\\
		S(2(2^{n-2}{+}k){-}1)=S\circ a^*(2^{n-2}{+}k{-}1))\oplus S(2^{n-2}{+}k).
		\end{aligned}
		\]
		
		From the induction hypothesis we have
		\[
		\begin{aligned}
		S(2(2^{n-1}{+}m)) &=S(2^{n-1}{+}m)\oplus S\circ a(2^{n-1}{+}m),\\
		S(2(2^{n-1}{+}m){-}1)&=S\circ a^*(2^{n-1}{+}m{-}1)\oplus S(2^{n-1}{+}m).
		\end{aligned}
		\]
		We take both an $\mathcal{L}_{\oplus}$ and an $\mathcal{R}_{\oplus}$ operation on either case in Equation (\ref{eq to start}). 		
		For $i=2k$ we have from Lemma \ref{lem must show} that
		\[
		\begin{aligned}
		\mathcal{L}_{\oplus}(S(2^{n-1}{+}i))&=S(2^{n-2}{+}k)\oplus S(2^{n-1}{+}2k)\\
		&=S\circ a^*(2^{n-1}{+}m{-}1)\oplus S(2^{n-1}{+}m),\\
		\mathcal{R}_{\oplus}(S(2^{n-1}{+}i))&=S(2(2^{n-2}{+}k))\oplus S\circ a(2^{n-2}{+}k)\\
		&=S(2^{n-1}{+}m)\oplus S\circ a(2^{n-1}{+}m).
		\end{aligned}
		\]
		
		For the case where $i=2k-1$ is similar after application of Lemma \ref{lem must show}. All cases for the element in the $n$-th level coming from Equation (\ref{eq to start}) are covered, completing the proof.
	\end{proof}
	
	\begin{definition}
		The length of the sequence $S(n)$ is denoted $|S(n)|$.
	\end{definition}
	
	\begin{remark}
		From Proposition \ref{seq def} we have since $|S(0)|=|S(1)|$ that
		\[
		\begin{aligned}
		|S(2n)|&=|S(n)|+|S\circ a(n)|,\ \mbox{and}\\
		|S(2n{-}1)|&=|S(a(n{-}1))|+ |S(n)|.
		\end{aligned}
		\]
		
	\end{remark}
	
	\subsection{Symmetry of construction of sequences $S((n))$}
	We use the symmetry of the graph $\mathcal{G}_{a,b},\prec $ and of Stern's diatomic sequence to prove Lemma \ref{lem new} that significantly shortens the proof of Proposition \ref{lem main}. For this we need the following short lemmas.	
	\begin{lemma} \label{lem1}
		For $k\geq 2$ we have
		\[
		|S(k)|=|S\circ a(k)|+|S\circ a(k{-}1)|.
		\]
	\end{lemma}
	
	\begin{proof}
		We prove this lemma by induction. First we have
		\[
		|S(2)|=4=2|S(1)|=|S\circ a(2)|+|S\circ a(1)|.
		\]
		
		Assume $|S(k)|=|S\circ a(k)|{+}|S\circ a(k{-}1)|$ for all $k=2,\ldots,N{-}1$, for some $N\in\Z$. We have two cases:
		\begin{itemize}
			\item[\emph{$N$ even}:] If $N=2m$, then
			\[
			\begin{aligned}
			|S(2m)|&=|S(m)|+|S\circ a(m)|\\
			&=|S\circ a(2m{-}1)|+|S\circ a(2m)|\\
			&=|S\circ a(N{-}1)|+|S\circ a(N)|.
			\end{aligned}
			\]
			\item[\emph{$N$ odd}:] The $N=2m-1$ case is similar. This concludes the proof.
			
		\end{itemize}
	\end{proof}
	
	\begin{lemma} \label{lem2}
		For $k\geq 1$ we have that
		\[
		|S\circ a(k)|=2d_k.
		\]
	\end{lemma}
	
	\begin{proof}
		We prove this by induction again. First we have
		\[
		|S\circ a(2)|=|S\circ a(1)|=2=2d_1=2d_2.
		\]
		Assume $|S\circ a(k)|=2d_k$ for all $k=1,\ldots,N{-}1$, for some positive integer $N$. We have two cases:
		\begin{itemize}
			\item[\emph{$N$ even}:] If $N=2m$, then $d_{2m}=d_m$, and $|S\circ a(2m)|=|S\circ a(m)|$. So
			\[
			2d_{N}=|S\circ a(N)|,
			\]
			which happens if and only if
			\[
			2d_m=|S\circ a(m)|.
			\]
			\item[\emph{$N$ odd}:] The $N=2m{-}1$ case is proved similarly. This concludes the proof.
		\end{itemize}
	\end{proof}
	
	\begin{lemma} \label{lem3}
		For a positive odd integer $k=k_1$, let $i\in\Z$ be such that the numbers
		\[
		k_j=\frac{k_{j-1}+1}{2},\  j=1,\ldots,i,
		\]
		are positive integers, with $k_i$ even. Let $k_{i+1}=k_i/2$. Then
		\[
		\frac{|S(k_{i+1})|}{2}=d_{k{-}1}.
		\]
	\end{lemma}
	
	\begin{proof}
		From Definition \ref{def a} and Lemma \ref{lem2} we have $a(2k_{i+1}{-}1){=}k_{i+1}$, $a(k_i{-}1){=}k_{i+1}$, and $d_{k_i{-}1}{=}d_{k_2{-}1}{=}\ldots{=}d_{k_i{-}1}$. Hence $S(k_{i+1}){=}S\circ a(k_i{-}1)$ and so
		\[
		|S\circ a(k_i{-}1)|=2d_{k{-}1}.
		\]
	\end{proof}
	
	\begin{lemma} \label{lem new}
		Let $n>1$. For $i=1,\ldots,2^{n-1}$ define the integers
		\[
		k'=6\cdot2^{n-2}{+}i{-}1\ \mbox{ and }\ k''=6\cdot2^{n-2}{-}i{+}1.
		\]
		Then we have
		\[
		\frac{\left|S(k'+1)\right|}{2}-d_{k'+1}=d_{k''}.
		\]
	\end{lemma}
	\begin{remark} \label{rem new}
		Let $f$ be a function taking a sequence $S_{A,B}(k)$ to $S_{B,A}(k)$.
		Due to the symmetry of the Definition \ref{def first s} we have that, for every $k>2$, there is a number $l$ such that
		\[
		S(k)=\overline{f(S(m))}.
		\]
		Lemma \ref{lem new} says that if $k=6\cdot2^{n-2}+i$ for positive integers $n$ and $i=1,\ldots,2^{n-1}$, then
		\[
		m=6\cdot2^{n-2}-i+1.
		\]
	\end{remark}
	
	\begin{proof}[Proof of Lemma \ref{lem new}]
		We have that $k'{+}1=a(2(k'{+}1){-}1)$ from Definition \ref{def a}.	Using Lemma \ref{lem2} we then have 
		\[
		\begin{aligned}
		\frac{|S(k'+1)|}{2}=\frac{|S\circ a(2(k'+1)-1)|}{2}&=d_{2(k'+1)-1}\\
		&=d_{k'+1}+d_{k'}.
		\end{aligned}
		\]
		It remains to show that
		\[
		d_{k'}=d_{k''},
		\]
		but this follows from the symmetry seen in Figure \ref{fig sds}.
	\end{proof}

	\subsection{Alternate form for Markov sequences} \label{ss ms path}
	In the proof of Proposition~\ref{lem main} we will use different formulae for Markov sequences than in Proposition \ref{seq def}, and we set these down in the following two Lemmas.
	
	\begin{lemma} \label{lem4}
		For a positive even integer $k=k_1$, let $i\in\Z$ be such that the numbers
		\[
		k_j=\frac{k_{j-1}}{2},\  j=1,\ldots,i,
		\]
		are positive integers, and $k_i$ is odd. Let $k_{i+1}=(k_i{+}1)/2$.
		Then
		\[
		S(k)=S\circ a^*(k_{i+1}{-}1)\oplus S(k_{i+1})^i,
		\]
		where the power $i$ indicates a sequence concatenated $i$ times.
	\end{lemma}
	
	\begin{proof}
		Through application of Proposition \ref{seq def} we have that
		\[
		S(k)=S(k_2)\oplus S\circ a(k_2)	=S(k_i)\oplus S\circ a(k_i)^{i-1}.
		\]
		Since $k_i=2k_{i+1}{-}1$, we have that
		\[
		S(k_i)=S\circ a^*(k_{i+1}{-}1)\oplus S(k_{i+1}).
		\]
		and so
		\[
		S(k)=S\circ a^*(k_{i+1}{-}1)\oplus S(k_{i+1})^i,
		\]
		proving the lemma.
	\end{proof}
	
	\begin{lemma} \label{lem5}
		For a positive odd integer $k=k_1$, let $i\in\Z$ be such that the numbers
		\[
		k_j=\frac{k_{j-1}+1}{2},\  j=1,\ldots,i,
		\]
		are positive integers, and $k_i=(k_{i-1}{+}1)/2$ is even. Let $k_{i+1}=k_i/2$. Then
		\[
		S(k)=\begin{cases}
		S(k_{i+1})^i\oplus S\circ a(k_{i+1}), & \mbox{if } k_i>2,\\
		S(0)^i\oplus S(1), & \mbox{if } k_i=2.
		\end{cases}
		\]
	\end{lemma}
	
	\begin{proof}
		Similar to Lemma \ref{lem4} the proof follows from application of Proposition \ref{seq def}.
	\end{proof}
	
	\begin{remark}
		We are never in a situation where
		\[
		S(k)=S(p)^\lambda\oplus S(q)^\rho,
		\]
		where $\lambda=\rho=1$,	since if $k$ is even and $k/2$ is odd, then $i=2$, and
		\[
		S(k)=S\circ a^*(k_{i+1}{-}1)\oplus S(k_{i+1})^2.
		\]
		A similar statement holds if $k$ is odd.
	\end{remark}
	
	Now we give the final lemma for Proposition \ref{lem main}.

	\begin{lemma} \label{lem6}
		Assume the $d_n$-th circular shift of $S(n)$ is palindromic for all $n=1,\ldots,k_1{-}1\in\Z$, and define $L=d_{k_2}$.
		\begin{itemize}
			\item[(\emph{i})] Let $k=k_1,\ldots,k_{i+1}$ be as in Lemma \ref{lem4}, for some positive integer $i$, and let
			\[
			R=L+\frac{|S\circ a(k_{i+1}{-}1)|+(i{-}1)|S(k_{i+1})|}{2}.
			\]
			Then $R>|S\circ a(k_{i+1}{-}1)|$.
			\item[(\emph{ii})] Let $k=k_1,\ldots,k_{j+1}$ be as in Lemma \ref{lem5}, for some positive integer $j$, and let
			\[
			R=L+\frac{|S(k_2)|}{2}.
			\]
			Then $L<(j{-}1)|S(k_{j+1})|$.
		\end{itemize}
	\end{lemma}
	
	\begin{proof}
		(\emph{i}) If $R>|S\circ a(k_{i+1}{-}1)|$ when $i=2$ then it is true for all $i>2$. So let $i=2$. We want to show
		\begin{equation} \label{eq2}
		d_{k_2}+\frac{|S\circ a(k_3{-}1)|+|S(k_2)|}{2}>|S\circ a(k_3{-}1)|.
		\end{equation}
		Since $|S\circ a(k_2)|/2=d_{k_2}$ from Lemma \ref{lem1}, and Lemma \ref{lem2}, we get that Equation (\ref{eq2}) becomes
		\[
		2|S\circ a(k_2)|+|S\circ a(k_2{-}1)|>|S\circ a(k_3{-}1)|,
		\]
		which is true, since
		\[
		a(k_3{-}1)=a(\tfrac{k_2+1}{2}{-}1)=a(\tfrac{k_2-1}{2})=a(k_2{-}1).
		\]
		
		\noindent (\emph{ii}) Consider the case were $k=k_1,\ldots,k_{j+1}$ are as in Lemma \ref{lem5}. Let $k^*_1$,\ldots,$k_j^*$ be as in Lemma \ref{lem5}, with $k_j^*=k_2$, and $k^*_{j-1}=k$. Since $k$ is odd, if $L<|S(k_2/2)|$ then
		\[
		L<|S(k_2/2)|<(j-1)|S(k_2/2)|=(j-1)|S(k_j^*/2)|.
		\]
		
		Let $k_3=k_2/2$. From Lemmas \ref{lem1} and \ref{lem2} we have that
		\[
		L=d_{k_2}=\frac{|S\circ a(k_3)|}{2},\quad \mbox{and}\quad 
		|S(k_3)|=|S\circ a(k_3)|+|S\circ a(k_3{-}1)|.
		\]
		From these two facts, we get
		\[
		L=\frac{|S\circ a(k_3)|}{2}<|S\circ a(k_3)|<|S\circ a(k_3)|{+}|S\circ a(k_3-1)|=|S(k_3)|,
		\]
		as required.
	\end{proof}
	
	\subsection{Proof of Proposition \ref{lem main}} \label{ssec proof}
	We give the proof of Proposition \ref{lem main}.
	
	\begin{proof}[Proof of Proposition \ref{lem main}]
		We prove this Proposition by induction on $n$. We must show two things: Firstly, that $S_{(a,,a),(b,b)}(n)$ is evenly palindromic, and secondly that $C_{d_n}(S(n))$ is palindromic. 
		
		It is clear that the statement holds for $n=0,1,2$. Assume the statement is true for all $n=3,\ldots,k{-}1$, for some $k>3$. We have two cases, for when $k$ is either even or odd. In either case, we denote the elements of the sequence $S(k)$ by $\lambda$'s, so
		\[
		S(k)=(\lambda_1,\ldots,\lambda_{|S(k)|}).
		\]
		
		\noindent(\emph{i}) Let $k=k_1$ be even. Let $i\in\Z$ be such that the numbers
		\[
		k_j=\frac{k_{j-1}}{2},\ j=1,\ldots,i,
		\]
		are positive integers, and $k_i$ is odd. Let $k_{i+1}=(k_i+1)/2$.
		Let $N$ and $M$ be the lengths of the sequences $S\circ a^*(k_{i+1}{-}1)$ and $S(k_{i+1})$ respectively. Then $		S(k)=S\circ a^*(k_{i+1}{-}1)\oplus S(k_{i+1})^i=(\lambda_1,\ldots,\lambda_{N{+}iM})$. 
		
		By the induction hypothesis we have that $C_{d_{k_2}}(S(k_2))$ is palindromic. Recall that $S(k)=S(k_2)\oplus S(k_{i+1})$.
		
		Let $L=d_{k_2}$, and set 
		\[
		R=L+\frac{N+(i-1)M}{2}.
		\]
		Using the fact that $R>N$ from Lemma \ref{lem6}, we already have the following relations for the elements of $S(k_2)$ 
		\[
		\lambda_L=\lambda_{L+1},\ldots,\lambda_{R}=\lambda_{R+1},
		\]
		and from this we derive the following relations for the elements of $S(k)$ 
		\[
		\lambda_L=\lambda_{L+1},\ldots,\lambda_{R}=\lambda_{R+M+1}.
		\]
		
		We must show the following condition
		\[
		\lambda_{R}=\lambda_{R+M+1},\ldots,\lambda_{R+M/2}=\lambda_{R+M/2+1}.
		\]
		To do this we note that, with $R{>}N$, we can ignore the sequence $S\circ a^*(k_{i+1}{-}1)$	at the start of $S(k)$, remove any excess copies of the sequence $S(k_{i+1})$, and get that this condition is equivalent to having
		\begin{equation} \label{eq case1}
		R{+}\frac{M}{2}{-}N \mod{M}=
		\begin{cases}
		d_{k_{i+1}}, &\mbox{if } i \mbox{ is even},\\
		d_{k_{i+1}}+\frac{|S(k_{i+1})|}{2}, &\mbox{if } i \mbox{ is odd}.
		\end{cases}
		\end{equation}
		Substituting in for $R$, $M$, and $N$, the left hand side of Equation (\ref{eq case1}) becomes
		\[
		d_{k_2}{+}\frac{|S(k_2)|}{2}{+}\frac{|S(k_{i+1})|}{2}{-}|S\circ a(k_{i+1}{-}1)|\mod{|S(k_{i+1})|}.	
		\]
		Recall that $d_{k_2}=\ldots=d_{2k_i-1}=d_{k_{i+1}}{+}d_{k_{i+1}{-}1}$ and
		\[
		|S(k_2)|=|S\circ a(k_{i+1}{-}1)|+(i{-}1)|S(k_{i+1})|,
		\]
		and so
		\[
		\begin{aligned}
		&R{+}\frac{M}{2}{-}N \mod{M}=\\
		&d_{k_{i+1}}+d_{k_{i+1}{-}1}+\frac{i|S(k_{i+1})|}{2}-\frac{|S\circ a(k_{i+1}{-}1)|}{2}\mod{M}.
		\end{aligned}
		\]
		If $i$ is even then this becomes
		\[
		d_{k_{i+1}}+d_{k_{i+1}{-}1}-\frac{|S\circ a(k_{i+1}-1)|}{2}\equiv d_{k_{i+1}} \mod{|S(k_{i+1})|},
		\]
		which is true by Lemma \ref{lem2}.
		
		If $i$ is odd, then we have 
		\[
		\begin{aligned}
		&d_{k_{i+1}}+d_{k_{i+1}{-}1}+\frac{|S(k_{i+1})|}{2}-\frac{|S\circ a(k_{i+1}-1)|}{2}\equiv\\
		&d_{k_{i+1}}+\frac{|S(k_{i+1})|}{2} \mod{|S(k_{i+1})|},
		\end{aligned}
		\]
		which is again true by Lemma \ref{lem2}. So we have that the $d_{k}$-th circular shift of $S(k)$ is palindromic, and the induction holds.

		\noindent (\emph{ii})
		The proof for the case when $k=k_1$ is odd is equivalent to the even case by Lemma \ref{lem new}, and Remark \ref{rem new},

		This concludes the proof of Proposition \ref{lem main}.
		
	\end{proof}

	\subsection{Final proof of Theorem \ref{thm main}} \label{ssec final proof}
	In this subsection we finalise the proof of Theorem \ref{thm main}. We start with the following Lemma, coming from \cite{cus1}.
	
	\begin{lemma} \label{lem 1st in proof}
		For each $k>1$ we have some $n$ such that
		\[
		S_{A,B}(k)=A^{\alpha_1}B^{\beta_1}\ldots A^{\alpha_n}B^{\beta_n}.
		\]
		From \cite{cus1} we have that if $A=(a,a)$ and $B=(b,b)$ then either $\alpha_i=1$ and $\beta_i\geq1$ for all $i$, with at least one $\beta_i>1$, or the opposite.
	\end{lemma}
	
	\begin{definition}
		Let $\Lambda=\overline{\Lambda}=(\lambda_1,\ldots,\lambda_m)$ for some positive $m$. Define the half sequences $\lfloor\Lambda\rfloor$ and $\lceil\Lambda\rceil$ by
		\[
		\lfloor\Lambda\rfloor=(\lambda_1,\ldots,\lambda_{\lfloor m/2\rfloor})\ \ \mbox{and}\ \ \lceil\Lambda\rceil= (\lambda_{\lceil m/2\rceil},\ldots,\lambda_m).
		\]
	\end{definition}
	\begin{remark}
		Clearly, $\overline{\lfloor\Lambda\rfloor}=\lceil\Lambda\rceil$.
	\end{remark}
	
	\begin{proof}[Proof of Theorem \ref{thm main}]
		Let $d_n$ be odd. For $A=(a,a)$, $B=(b,b)$, and $\lambda_1=\lambda_2\in\{a,b\}$ we have that $(\lambda_1)\Lambda_{\lceil d_n/2\rceil+1}\ldots\Lambda_{\lfloor d_n/2\rfloor-1}(\lambda_2)$ is palindromic by Lemma \ref{lem 1st in proof} and Proposition \ref{lem main}. Replacing $A$ and $B$ in $\Lambda_{\lceil d_n/2\rceil+1}\ldots\Lambda_{\lfloor d_n/2\rfloor-1}$ with any two palindromic sequences does not change this fact, neither does letting $\lambda_1=\lceil \Lambda\rceil$, $\lambda_2=\lfloor \Lambda\rfloor$ for any palindromic $\Lambda$.
		
		If $d_n$ is even, then $\Lambda_{d_n/2}=\Lambda_{-1+d_n/2}$, by Lemma \ref{lem 1st in proof} and Proposition \ref{lem main}. The result follows as a corollary of Lemma \ref{lem 1st in proof}.		
	\end{proof}

	\bibliographystyle{plain}
	\bibliography{biblio}
	
	\vspace{.5cm}
	
\end{document}